\documentclass[12pt]{article}

\usepackage[utf8]{inputenc}
\usepackage{tikz}

\usepackage{amsmath, amsfonts,amssymb,  amscd, amsthm, verbatim, mathtools, hyperref} 
\usepackage{fullpage, verbatim}
\usepackage[all]{xy}

\def\A{\mathbb{A}}

\def\F{\mathbb{F}}

\theoremstyle{plain}

\newtheorem{thm}{Theorem}
\newtheorem{Lem}[thm]{Lemma}
\newtheorem{Cor}[thm]{Corollary}

\theoremstyle{remark}

\renewcommand\footnotemark{}

\providecommand{\keywords}[1]{{\textit{Keywords:}} #1}

\providecommand{\MSC}[1]{{\textit{MSC codes:}} #1}

\begin{document}

\title{An application of random plane slicing to counting 
$\F_q$--points on hypersurfaces}
\author{Kaloyan Slavov\thanks{This research was
supported by NCCR SwissMAP of the SNSF.}}

\maketitle

\begin{abstract}
Let $X$ be an absolutely irreducible hypersurface of degree $d$ in $\A^n$, defined over a finite field $\F_q$. The Lang--Weil bound gives an interval that contains $\#X(\F_q)$. We exhibit explicit intervals which do {\it not} contain $\#X(\F_q)$, and which overlap with the Lang--Weil interval. In particular, we sharpen the best known lower and upper bounds for 
$\#X(\F_q)$.  The proof uses a 
probabilistic combinatorial technique.

\end{abstract}


\MSC{14G15, 14J70, 11G25}

\section{Introduction}

Let $\F_q$ be the finite field with $q$ elements and let $\A^n$ be the affine space over $\F_q.$ 
Let $X\subset\A^n$ (where $n\geq 3$) be a geometrically irreducible hypersurface of degree $d$, defined over $\F_q.$ 

We recall the following 

\begin{thm}[Weil, \cite{Weil}] Let $X$ be an absolutely irreducible plane curve of degree $d$. Then
\begin{equation}
|\#X(\F_q)-q|\leq (d-1)(d-2)\sqrt{q}+d+1.
\label{LW_curve}
\end{equation}
\label{Weil_curve}
\end{thm}
In the higher-dimensional case, Lang and Weil \cite{Lang_Weil} have established the bound
\[|\#X(\F_q)-q^{n-1}|\leq (d-1)(d-2)q^{n-3/2}+O_d(q^{n-2}).\]
There have been various results and improvements on the implied constant. 
In the case $n=3$, the best explicit bound that we were able to find in the literature is the $n=3$ case of the
following outcome of advanced
$l$-adic \'{e}tale cohomology techniques:

\begin{thm}[Ghorpade \& Lachaud, \cite{GL}]
Let $X\subset\A^n$ be an absolutely irreducible hypersurface over $\F_q$ of degree $d$. Then
\begin{equation}
|\#X(\F_q)-q^{n-1}|\leq (d-1)(d-2)q^{n-3/2}+12(d+3)^{n+1}q^{n-2}.
\label{LW_surface}
\end{equation}
\end{thm}

The best explicit bound for $n\geq 4$ that we are aware of is given by 

\begin{thm}[Cafure \& Matera, \cite{Cafure_Matera}]
For an absolutely irreducible hypersurface $X\subset\A^n$ of degree $d$ over $\F_q$,
\begin{equation}
|\#X(\F_q)-q^{n-1}|\leq (d-1)(d-2)q^{n-3/2}+5d^{13/3}q^{n-2}.
\label{LW_general}
\end{equation}
Moreover, if $q>15d^{13/3},$ then
\begin{equation}
|\#X(\F_q)-q^{n-1}|\leq (d-1)(d-2)q^{n-3/2}+
(5d^2+d+1)q^{n-2}.
\label{LW_regularity}
\end{equation}

\label{LW_bound}
\end{thm}

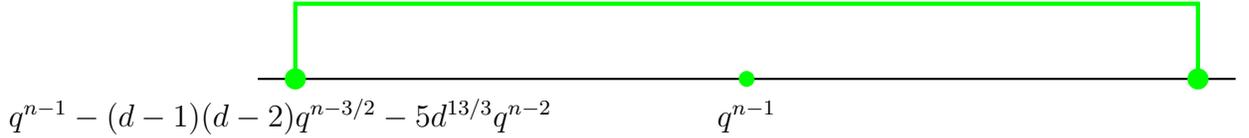
\hspace*{-0.6cm}
\begin{tikzpicture}
\draw[thick](-6.5,0)--(6.5,0);
\fill[green] (0,0)circle(3pt);
\node at (0,-0.5) {$q^{n-1}$};
\fill[green] (-6,0)circle(4pt);
\fill[green] (6,0)circle(4pt);
\draw[green, ultra thick](-6,0)--(-6,1)--(6,1)--(6,0);
\node at (-6.2,-0.5) {
$q^{n-1} -(d-1)(d-2)q^{n-3/2}
-5d^{13/3}q^{n-2}$
};
\end{tikzpicture}

\bigskip

The idea of Cafure and Matera is to estimate the number of two--dimensional affine planes $H\subset\A^n_{\F_q}$ for which the intersection $X\cap H$ has a given number of geometrically irreducible 
$\F_q$--components, and apply the Lang--Weil bound coming from the one--dimensional case (\ref{LW_curve}) for each such $X\cap H$. They refine the theorem of Kaltofen
(see \cite{Kaltofen}), which states that for 
$q>\frac{3}{2}d^4-2d^3+\frac{5}{2}d^2,$ there exists a plane $H$ such that $X\cap H$ is geometrically irreducible, by keeping track of the actual number of
geometrically irreducible  components.   

We combine the main idea of \cite{Cafure_Matera} with the technique of ``random plane slicing'' from
\cite{Tao_blog} and prove the following

\begin{thm} 
Let $X\subset\A^n$ (with $n\geq 3$) be a 
hypersurface over $\F_q$ of degree $d.$ 
Let $N=\#X(\F_q).$

Consider parameters $0<\alpha<1$ and $e>0$ such that
$e\alpha^2>1.$
If \[N>\max\left(\frac{d^2q^{n-2}}{4(1-\alpha)}, 
\frac{e}{e\alpha^2-1}q^{n-2}
\right),\]
then in fact
\[N>q^{n-1}-(d-1)(d-2)q^{n-3/2}-(d+1+e)q^{n-2}.\]
\label{values_skip}
\end{thm}

Notice that in case $d$ is large, namely
\[d\geq 2\left(\frac{e(1-\alpha)}{e\alpha^2-1}\right)^{1/2},\]
 the condition on $N$ in the theorem becomes simply 
 $N>\dfrac{d^2q^{n-2}}{4(1-\alpha)}.$
  
It is enlightening to write down the statement of Theorem \ref{values_skip} for a concrete choice of the parameters: for example, when 
$\alpha=2/3, e=9,$ we can state the following

\begin{Cor}
Let $X\subset\A^n$ ($n\geq 3$) be a hypersurface over $\F_q$ of degree 
$d.$ Let $N=\#X(\F_q).$ 
\[\text{If\ } N>\frac{3}{4}d^2q^{n-2},
\ \text{then in fact}\quad 
N>q^{n-1}-(d-1)(d-2)q^{n-3/2}-(d+10)q^{n-2}.\]
\label{cor_explicit}
\end{Cor}

In other words, $N$ cannot belong to the interval

\bigskip

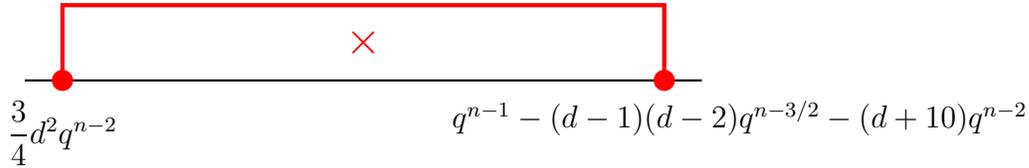
\begin{tikzpicture}
\draw[thick](-5.5,0)--(3.5,0);
\node at (-1,0.5){{\Large\color{red}{$\times$}}};
\fill[red] (-5,0)circle(4pt);
\fill[red] (3,0)circle(4pt);
\draw[red, ultra thick](-5,0)--(-5,1)--(3,1)--(3,0);
\node at (-5,-0.7){
$\dfrac{3}{4}d^2q^{n-2}$
};
\node at (4,-0.5) {
$q^{n-1} -(d-1)(d-2)q^{n-3/2}
-(d+10)q^{n-2}$
};
\end{tikzpicture}

\bigskip

Suppose $X$ is geometrically irreducible. Then this forbidden interval overlaps with the best known Lang--Weil intervals in the various ranges for $q$, as summarized in the diagram below (we analyze the case $n\geq 4$; the pictures for $n=3$ are similar). When we can afford a slight loss of precision, we write $g(d)+\ldots$ for $g(d)+o(g(d))$. Note that the content of Corollary \ref{cor_explicit} concerns the regime $q\geq d^4+\dots.$

\bigskip

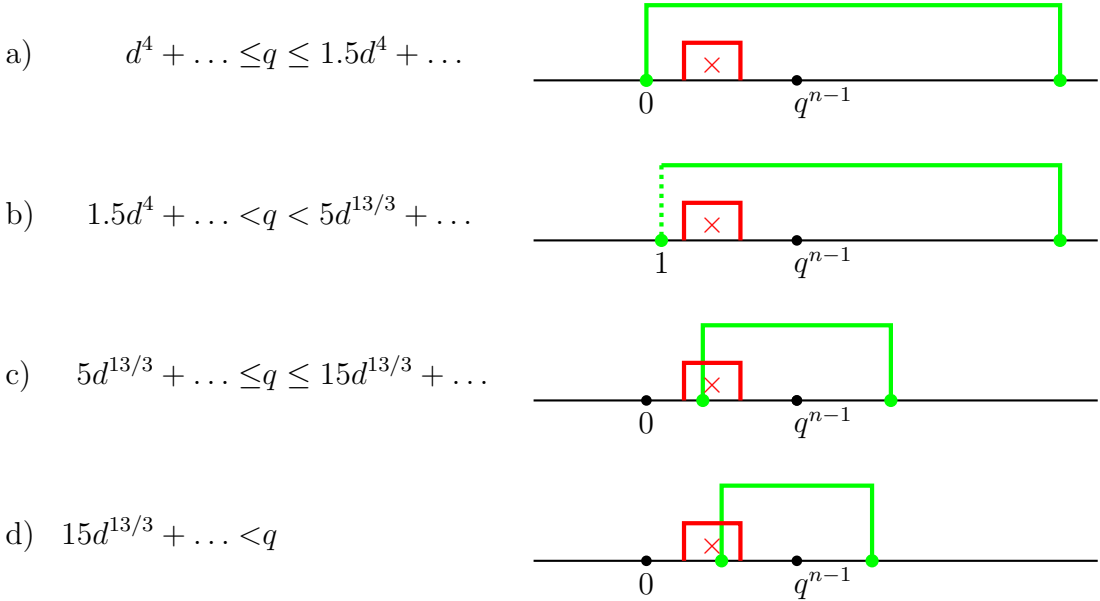
\begin{minipage}[t]{0.35\textwidth}
\vspace*{-1.8cm}
\begin{align*}
&{\text{a)}} &d^4+\ldots\leq &q\leq 1.5d^4+\ldots
\\[8ex]
&{\text{b)}} &1.5d^4+\ldots < &q< 5 d^{13/3}+\ldots\\[8ex]
&{\text{c)}} &5d^{13/3}+\ldots\leq &q\leq 15d^{13/3}+\ldots\\[8ex]
&{\text{d)}} &15d^{13/3}+\ldots< &q
\end{align*}
\end{minipage}%
\hfill
\hspace{0.5cm}
\begin{minipage}[t]{0.65\textwidth}
\begin{tikzpicture}[scale=0.5]
\draw[thick](-7,0)--(8,0);
\draw[green, ultra thick](-4,0)--(-4,2)--(7,2)--(7,0);
\fill[green](-4,0) circle (5pt);
\node at (-4,-0.6) {$0$};
\fill[green](7,0) circle (5pt);
\fill[black](0,0) circle(4pt);
\node at (0.7,-0.6) {$q^{n-1}$};

\draw[red, ultra thick](-3,0)--(-3,1)--(-1.5,1)--(-1.5,0);
\node at (-2.25,0.4){\color{red}{$\times$}};
\end{tikzpicture}

\bigskip

\begin{tikzpicture}[scale=0.5]

\draw[thick](-7,0)--(8,0);
\draw[green, ultra thick,dotted](-3.6,2)--(-3.6,0);
\draw[green, ultra thick](-3.6,2)--(7,2)--(7,0);
\fill[green](-3.6,0) circle (5pt);
\node at (-3.6,-0.6) {$1$};
\fill[green](7,0) circle (5pt);
\fill[black](0,0) circle(4pt);
\node at (0.7,-0.6) {$q^{n-1}$};

\draw[red, ultra thick](-3,0)--(-3,1)--(-1.5,1)--(-1.5,0);
\node at (-2.25,0.4){\color{red}{$\times$}};
\end{tikzpicture}

\bigskip

\begin{tikzpicture}[scale=0.5]
\draw[thick](-7,0)--(8,0);
\draw[green, ultra thick](-2.5,0)--(-2.5,2)--(2.5,2)--(2.5,0);
\fill[green](-2.5,0) circle (5pt);
\fill[green](2.5,0) circle (5pt);
\fill[black](0,0) circle(4pt);

\fill[black](0,0) circle(4pt);
\node at (0.7,-0.6) {$q^{n-1}$};

\fill[black](-4,0) circle (4pt);
\node at (-4,-0.6) {$0$};

\draw[red, ultra thick](-3,0)--(-3,1)--(-1.5,1)--(-1.5,0);
\node at (-2.25,0.4){\color{red}{$\times$}};
\end{tikzpicture}

\bigskip

\begin{tikzpicture}[scale=0.5]
\draw[thick](-7,0)--(8,0);
\draw[green, ultra thick](-2,0)--(-2,2)--(2,2)--(2,0);
\fill[green](-2,0) circle (5pt);
\fill[green](2,0) circle (5pt);
\fill[black](0,0) circle(4pt);
\node at (0.7,-0.6) {$q^{n-1}$};
\fill[black](-4,0) circle (4pt);
\node at (-4,-0.6) {$0$};

\draw[red, ultra thick](-3,0)--(-3,1)--(-1.5,1)--(-1.5,0);
\node at (-2.25,0.4){\color{red}{$\times$}};
\end{tikzpicture}

\end{minipage}

 \bigskip
 
In range a), the left endpoint of the Lang-Weil interval is $0$, since we do not know any nontrivial lower bound for $N$, when
 $q<\frac{3}{2}d^4-2d^3+\frac{5}{2}d^2$. In range b), 
the lower bound coming from (\ref{LW_general}) is vacuous, but  
Theorem 5.4 in \cite{Cafure_Matera} implies\footnote{The authors state a hypothesis $q>2d^4$ but their proof works in fact for $q>1.5d^4+\dots$. The dotted vertical line in our diagram stands for the fact that the argument in \cite{Cafure_Matera} actually gives a slightly better lower bound than the stated $N\geq 1$.} that $N\geq 1$.  In the ranges c) and d), the best known Lang--Weil intervals are given by (\ref{LW_general}) and (\ref{LW_regularity}), respectively.
We formulate the precise statement describing cases c) and d) as the following

\begin{Cor}
Let $X\subset\A^n$ ($n\geq 3$) be an absolutely irreducible hypersurface over $\F_q$ of degree 
$d.$ Let $N=\#X(\F_q).$ Suppose that
\[q\geq \frac{1}{4}\left((d-1)(d-2)+\sqrt{(d-1)^2(d-2)^2+20d^{13/3}+3d^2} \right)^2.\]
Then
\[N>q^{n-1}-(d-1)(d-2)q^{n-3/2}-(d+10)q^{n-2}.\] 	
\end{Cor}  

\begin{proof}
	For such values of $q$, (\ref{LW_general}) implies that $N>\frac{3}{4}d^2q^{n-2}.$ 
\end{proof}
 
Compare this estimate with the lower bounds coming from (\ref{LW_general}) and (\ref{LW_regularity}). While we sharpen the best known nontrivial lower bounds for $N$, we find the existence of an exclusion zone in cases a) and b) no less interesting.

As an application, the bound in Corollary \ref{cor_explicit} sharpens Theorem 2 in \cite{Zahid}, concerning the existence of a smooth point on a hypersurface over $\F_q$.

\begin{Cor}
Let $G\in\F_q[x_1,...,x_n]$ be an absolutely irreducible polynomial of degree $d$, and let $H\in\F_q[x_1,...,x_n]$ be a polynomial of degree $e$, not divisible by $G$. Then there exists a nonsingular zero of $G$, which is not a zero of $H$, provided that
\[q>\frac{1}{4}\left((d-1)(d-2)+\sqrt{(d-1)^2(d-2)^2+4(d^2+de+10)}\right)^2.\]
\end{Cor}

\begin{proof}
In the proof of Theorem 2 in \cite{Zahid}, replace the bound coming from the Cafure--Matera estimate (\ref{LW_general}) by the estimate from Corollary \ref{cor_explicit}.
\end{proof}

We prove a similar result for an ``upper'' forbidden interval:

\begin{thm}
Let $X\subset\A^n_{\F_q}$ be a hypersurface of degree $d$ and let 
$N=|X(\F_q)|$. Consider parameters $\alpha,\gamma\in (0,1), A,B>0$. If
\begin{multline*}
N<\min( (2-\gamma)q^{n-1}-(d-2)(d-3)q^{n-3/2}-(d^2/4+d+2)q^{n-2},\\ B\gamma^2q^{n-1},\alpha^2Aq^{n-1}, (1-\alpha)q^{n+1}),
\end{multline*}
then in fact
\[N<q^{n-1}+(d-1)(d-2)q^{n-3/2}+((B+1)d+1+A/q)q^{n-2}.\]
\label{values_skip_upper}
\end{thm}

For example, if $\alpha=\gamma=2/3, A=B=3$, we can state

\begin{Cor}
Let $X\subset\A^n_{\F_q}$ be a hypersurface of degree $d$ and let 
$N=|X(\F_q)|$. If
\[N<\frac{4}{3}q^{n-1}-(d-2)(d-3)q^{n-3/2}-\left(\frac{d^2}{4}+d+2\right)q^{n-2},\]
then in fact
\[N<q^{n-1}+(d-1)(d-2)q^{n-3/2}+(4d+1+3/q)q^{n-2}.\]
\label{Cor_upper}
\end{Cor}

When $X$ is geometrically irreducible, the pictures that visualize this statement are similar to the ones we have discussed above in various ranges for $q$. In particular, when $q$ is large relative to $d$, this improves the best known upper bound for $|X(\F_q)|$. 

The strategy for proving Theorem \ref{values_skip} is as follows. We intersect $X$ with random planes 
$H$ defined over $\F_q.$ Each slice $X\cap H$ satisfies a dichotomy property: it either has a component which is
an absolutely irreducible plane curve of degree at most $d$, hence contains
plenty of $\F_q$--rational points by (\ref{LW_curve}), or $X\cap H$ contains very few such. If we assume that $\#X(\F_q)$ is ``reasonably'' large from the onset, the mean of the random variable $|(X\cap H)(\F_q)|$ is ``reasonably'' large as well. The variance bound in \cite{Tao_blog} implies that {\it plenty} of values    
$|(X\cap H)(\F_q)|$ must be concentrated close to this reasonably large mean, hence, by the dichotomy property, have to be in fact {\it large}. Thus, {\it many} planes will have {\it large} intersections with $X(\F_q)$. Adding up their contributions refines the initial bound on 
$\#X(\F_q)$. The details are spelled out in Section
\ref{section_lower_interval}. The proof of Theorem \ref{values_skip_upper} is analogous but requires a few twists that we discuss in Section \ref{section_upper}.

The reason a reduction to the case of a plane curve is so appealing (also in \cite{Cafure_Matera}) is that the 
Lang--Weil bound (\ref{LW_curve}) in the case of a plane curve has a distinctive advantage over the bounds
(\ref{LW_surface}) or (\ref{LW_general}) in the higher--dimensional case: in order for the power $q^{\dim X}$ that approximates $\#X(\F_q)$ to dominate the error term, it takes only that $q>d^4+\dots$ in the case $\dim X=1$, rather than $q>12d^4+\dots$ in the case $\dim X=2$ or $q>5d^{13/3}+\dots$ in the case $n>3.$

\section{Proof of the result for the ``lower'' forbidden interval}
\label{section_lower_interval}

The crucial technique in the proof is the random sampling method from Section 2 in \cite{Tao_blog}. For us, a ``plane'' will mean a $2$-dimensional affine linear subspace of $\A^n_{\F_q}$ or $\F_q^n$, depending on the context. We need a variant of 
Lemma 7 in \cite{Tao_blog} for planes rather than hyperplanes. 

\begin{Lem}
Let $E$ be a subset of $\F_q^n$ of cardinality $N$. For a plane $H\subset\F_q^n$ chosen uniformly at random, consider the cardinality $|E\cap H|$ as a random variable. 
\begin{description}
\item[a)] The mean of $|E\cap H|$ is $\mu=\dfrac{N}{q^{n-2}}$.
\item[b)] The variance $\sigma^2$ of $|E\cap H|$ satisfies
\[\sigma^2\leq \frac{N}{q^{n-2}}.\]
\end{description}
\label{variance_bound}
\end{Lem}

\begin{proof}
We modify the proof of Lemma 7 in \cite{Tao_blog}, just replacing hyperplanes by planes. The statement for the mean follows from the fact that any point of $\F_q^n$ belongs to exactly $\dfrac{1}{q^{n-2}}$ of all the planes. Indeed, the map
\[\{\text{planes in $\F_q^n$}\}\to\{\text{planes in $\F_q^n$ containing $0$}\}\]
sending a plane $U$ to the unique translate $U_0$ of $U$ containing $0$ is $q^{n-2}:1$, since the  translates of some $U_0$ through $0$ are exactly the planes $v+U_0,$ where $v$ belongs to an $(n-2)$-dimensional subspace of $\F_q^n$, complementary to $U_0$. 

For points $a\neq b$ in $\F_q^n$, we count the number of planes containing both $a$ and $b$. WLOG, $a=0$. The number of planes containing the line through $0$ and $b$ equals $\dfrac{q^n-q}{q^2-q}$. The number of planes through $0$ equals $\dfrac{(q^n-1)(q^n-q)}{(q^2-1)(q^2-q)}.$ Finally, the proportion of planes through $0$ and $b$ to the total number of planes is
\[\frac{q^2-1}{q^{2n-2}-q^{n-2}}\leq\frac{1}{q^{2n-4}}.\]
Therefore,
\begin{align*}
\mu^2+\sigma^2=\mathbb{E}|E\cap H|^2 &=\mu+|E|\left(|E|-1\right)\frac{\#\text{planes through two distinct points}}{\text{total number of planes}}\\
&\leq\mu+|E|(|E|-1)\frac{1}{q^{2n-4}}\\
&\leq \mu+\mu^2.\qedhere
\end{align*}
\end{proof}

A plane slice of a hypersurface satisfies the following dichotomy property. 

\begin{Lem}
Let $X\subset\A^n$ be a hypersurface of degree $d$, defined over $\F_q.$
Let $H$ be a plane defined over $\F_q.$ Consider the intersection $X\cap H.$ Then either
\begin{itemize}
\item[a)] $|(X\cap H)(\F_q)|\geq q-(d-1)(d-2)\sqrt{q}-d-1$, or
\item[b)] $|(X\cap H)(\F_q)|\leq\dfrac{d^2}{4}.$ 
\end{itemize}
\label{dichotomy}
\end{Lem}

\begin{proof}
The statement is clear in the cases $H\subset X$ or $X\cap H=\emptyset$. In the generic case $\dim X\cap H=1$, let $X_1,...,X_s$ be the $\F_q$--components of $X\cap H$.
Each $X_i$ is $1$--dimensional and has degree at most $d$. 

If $X_i$ is geometrically irreducible for some $i$, we apply the lower bound from (\ref{LW_curve}). 

Suppose that no $X_i$ is geometrically irreducible. Let $d_i=\deg(X_i).$ By Lemma 2.3 in \cite{Cafure_Matera}, 
\[\#X_i(\F_q)\leq \frac{d_i^2}{4}.\]
Therefore,
\[|(X\cap H)(\F_q)|\leq\sum_{i=1}^s\frac{d_i^2}{4}\leq\frac{d^2}{4}.\qedhere\]
\end{proof}

We say that a plane $H\subset\A^n$, defined over $\F_q$, is ``bad,'' if $X\cap H$ satisfies b), and otherwise we say that $H$ is good. 

\begin{Lem}
For $H$ chosen uniformly at random among planes over 
$\F_q$,
\[\text{Prob}(H\text{\ is bad})\leq\dfrac{\dfrac{N}{q^{n-2}}}{\left(\dfrac{N}{q^{n-2}}-\dfrac{d^2}{4} \right)^2}.\]
\label{prob_bad_hplane}
\end{Lem}

\begin{proof} 
If $H$ is a bad hyperplane, 
\[
\left|\#(X\cap H)(\F_q)-\frac{N}{q^{n-2}}\right|\geq \frac{N}{q^{n-2}}-\frac{d^2}{4}.
\]
Define $k$ so that 
\[\frac{N}{q^{n-2}}-\frac{d^2}{4}=k\sigma.\]
The bound for $\sigma^2$ from Lemma \ref{variance_bound} yields 
\[k\geq \dfrac{\frac{N}{q^{n-2}} -\frac{d^2}{4}}{\sqrt{\frac{N}{q^{n-2}}}},\]
and the statement follows from Chebyshev's inequality
\[\text{Prob}\left(\left|\#(X\cap H)(\F_q)-\frac{N}{q^{n-2}}\right|\geq k\sigma\right)\leq\frac{1}{k^2}.\qedhere\]
\end{proof}

\begin{proof}[Proof of Theorem \ref{values_skip}]
By the condition on $N$, we have
$\dfrac{N}{q^{n-2}}-\dfrac{d^2}{4}> \alpha\dfrac{N}{q^{n-2}},$
hence Lemma \ref{prob_bad_hplane} implies
\[\text{Prob} (H\text{\ is bad})\leq \frac{\dfrac{N}{q^{n-2}}}{\left(\dfrac{N}{q^{n-2}}-\dfrac{d^2}{4}\right)^2}<
\frac{N/q^{n-2}}{\alpha^2(N/q^{n-2})^2}=\frac{q^{n-2}}{\alpha^2 N}.\] 
Therefore,
\[\text{Prob} (H\text{\ is good})> 1-\dfrac{q^{n-2}}{\alpha^2N}=\frac{\alpha^2N-q^{n-2}}{\alpha^2 N}.\]

Therefore, the mean $\mu$ of the random variable $|(X\cap H)(\F_q)|$ satisfies
\begin{align*}
\frac{N}{q^{n-2}}=\mu &\geq \text{Prob}(H\text{\ is good})\left(q-(d-1)(d-2)\sqrt{q}-d-1\right) \\
&> \frac{\alpha^2N-q^{n-2}}{\alpha^2 N}
\left(q-(d-1)(d-2)\sqrt{q}-d-1\right).
\end{align*}
The initial assumption on $N$ yields the first inequality in the chain below:
\[N+eq^{n-2}>\frac{\alpha^2 N^2}{\alpha^2 N-q^{n-2}}>
q^{n-1}-(d-1)(d-2)q^{n-3/2}-(d+1)q^{n-2}.\qedhere\]
\end{proof}

\section{Proof of the ``upper'' exclusion zone}
\label{section_upper}

\begin{Lem} Let $X\subset\A^n_{\F_q}$ be a hypersurface of degree $d$ and let $H\subset\A^n_{\F_q}$ be a plane. Then either
\begin{itemize}
\item[a)] $|(X\cap H)(\F_q)|\leq q+(d-1)(d-2)\sqrt{q}+d+1$, or
\item[b)] $|(X\cap H)(\F_q)|\geq 2q-(d-2)(d-3)\sqrt{q}-\left(d^2/4+d+2\right).$ 
\end{itemize}
\label{dichotomy_upper_bound}
\end{Lem}

\begin{proof}
If $X\cap H=\emptyset$ or $H\subset X$, the statement is clear. In the generic case, $X\cap H=X_1\cup\dots\cup X_s,$ where $X_1,...,X_s$ are the $\F_q$-components of $X\cap H$, with $\dim X_i=1$ for each $i$. Let $d_i=\deg(X_i)$. 

If no $X_i$ is geometrically irreducible, the proof of Lemma \ref{dichotomy} implies the first inequality below (and the second one is immediately checked directly):
\[|(X\cap H)(\F_q)|\leq \frac{d^2}{4}\leq q+(d-1)(d-2)\sqrt{q}+d+1.\]

Suppose that exactly one among the $X_i$'s is geometrically irreducible, and let $e$ denote its degree, $1\leq e\leq d$. Then 
\begin{align*}
|(X\cap H)(\F_q)| &\leq q+(e-1)(e-2)\sqrt{q}+e+1+\frac{(d-e)^2}{4}\\
&\leq q+(d-1)(d-2)\sqrt{q}+d+1,
\end{align*}
where the former inequality in the chain follows from combining the bound 
(\ref{LW_curve}) applied to the geometrically irreducible component of $X\cap H$ and the proof of Lemma \ref{dichotomy} applied to the remaining components, while the latter inequality is a matter of direct verification. 

Suppose now that at least two $\F_q$-components (say $X_1$ and $X_2$) of $X\cap H$ are geometrically irreducible. Let $e_1$ and $e_2$ denote their degrees, $e_1+e_2\leq d.$ Then 
\[|(X_1\cap X_2)(\F_q)|\leq e_1e_2\]
by B\'ezout's theorem and therefore the bound (\ref{LW_curve}) applied to $X_1$ and $X_2$ implies
\begin{align*}
|(X\cap H)(\F_q)|& \geq |X_1(\F_q)|+|X_2(\F_q)|-|(X_1\cap X_2)(\F_q)|\\
&\geq 2q -(e_1-1)(e_1-2)\sqrt{q}-(e_2-1)(e_2-2)\sqrt{q}-(e_1+e_2)-2-e_1e_2\\
&\geq 2q-(d-2)(d-3)\sqrt{q}-d-2-\frac{d^2}{4}.\qedhere
\end{align*}
\end{proof}



In this section, a plane $H$  will be called ``bad'' if it satisfies condition b) above, and ``good'' otherwise. A plane $H$ is ``very bad'' if 
$H\subset X$, in which case $|(X\cap H)(\F_q)|=q^2$. 

\begin{Lem}
Let $X\subset\A^n_{\F_q}$ be a hypersurface of degree $d$ and let $H\subset\A^n_{\F_q}$ be a plane such that 
$H\not\subset X$. Then
\[|(X\cap H)(\F_q)|\leq dq.\]
\end{Lem}

\begin{proof}
This follows from the Schwartz--Zippel lemma applied to the degree-$d$ plane curve $X\cap H\subsetneq H$.  
\end{proof}

\begin{proof}[Proof of Theorem \ref{values_skip_upper}]
Just as in Lemma \ref{prob_bad_hplane}, we have
\begin{align*}
\text{Prob}(H\text{\ is bad}) &\leq 
\frac{N/q^{n-2}}{\left(2q-(d-2)(d-3)\sqrt{q}-(d^2/4+d+2)-N/q^{n-2}   \right)^2}\\
\text{Prob}(H\text{\ is very bad}) &\leq 
\frac{N/q^{n-2}}{\left(q^2-N/q^{n-2}\right)^2}.
\end{align*}

We now add up the contributions of all $|(X\cap H)(\F_q)|$ and write
\begin{multline*}
\frac{N}{q^{n-2}}=\mu\leq 
q+(d-1)(d-2)\sqrt{q}+d+1+
\text{Prob(H \text{\ is bad})}dq
+
\text{Prob(H \text{\ is very bad})}q^2.
\end{multline*}
This yields
\begin{multline*}
N\leq q^{n-1}+(d-1)(d-2)q^{n-3/2}+(d+1)q^{n-2}+\\
\frac{Ndq}{(2q-(d-2)(d-3)\sqrt{q}-(d^2/4+d+2)-N/q^{n-2})^2}+\frac{Nq^2}{(q^2-N/q^{n-1})^2}
\end{multline*}
and so we are left to bound the last two terms above.
By the assumptions on $N$, we have 

\begin{align*}
\frac{Ndq}{(2q-(d-2)(d-3)\sqrt{q}-(d^2/4+d+2)-N/q^{n-2})^2}&\leq \frac{Ndq}{\gamma^2q^2}\leq Bdq^{n-2}\quad\text{and}\\
\frac{Nq^2}{(q^2-N/q^{n-1})^2} &\leq 
\frac{Nq^2}{\alpha^2q^4}\leq Aq^{n-3}.\qedhere
\end{align*}
\end{proof}

\vspace{+4 pt}
\noindent Department of Mathematics \\
\noindent ETH Z\"urich \hfill  \\
\noindent kaloyan.slavov@math.ethz.ch

\end{document}